\documentclass[sigconf]{acmart}

\usepackage{amsmath, amsthm, color}

\pdfoutput=1
\usepackage{graphicx}
\usepackage{caption}
\usepackage{mathtools}
\usepackage{enumitem}
\usepackage{xfrac}
\usepackage{verbatim}
\usepackage{tikz}
\usetikzlibrary{matrix}
\usetikzlibrary{arrows}
\usepackage{algorithm}
\usepackage[noend]{algpseudocode}
\usepackage{caption}
\usepackage{subcaption}
\oddsidemargin \evensidemargin
\usepackage{makecell}
\usepackage{array}

\newtheorem{theorem}{Theorem}
\newtheorem{proposition}[theorem]{Proposition}
\newtheorem{lemma}[theorem]{Lemma}
\newtheorem{corollary}[theorem]{Corollary}

\theoremstyle{definition}

\newtheorem{remark}[theorem]{Remark}
\newtheorem{conjecture}[theorem]{Conjecture}

\newtheorem{example}[theorem]{Example}

\newcommand{\PP}{\mathbb{P}}

\title{\bf Moment Varieties for Mixtures of Products}
\date{}

\copyrightyear{2023}
\acmYear{2023}
\setcopyright{rightsretained}
\acmConference[ISSAC 2023]{International Symposium on Symbolic and Algebraic Computation 2023}{July 24--27, 2023}{Tromsø, Norway}
\acmBooktitle{International Symposium on Symbolic and Algebraic Computation 2023 (ISSAC 2023), July 24--27, 2023, Tromsø, Norway}\acmDOI{10.1145/3597066.3597084}
\acmISBN{979-8-4007-0039-2/23/07}

\begin{document}

\author{Yulia Alexandr}
\email{yulia@math.berkeley.edu}
\affiliation{%
  \institution{UC Berkeley}
  \city{Berkeley}
  \country{California, USA}}

\author{Joe Kileel}
\email{jkileel@math.utexas.edu}
\affiliation{%
  \institution{UT Austin}
  \city{Austin}
  \country{Texas, USA}}

\author{Bernd Sturmfels}
\email{bernd@mis.mpg.de}
\affiliation{%
  \institution{MPI-MiS Leipzig  and UC Berkeley}
  \city{Leipzig}
  \country{Germany}}

\begin{abstract}
The setting of this article is nonparametric algebraic statistics. We study moment varieties of
 conditionally independent mixture distributions on $\mathbb{R}^n$. These are the secant varieties of
 toric varieties that express independence in terms of univariate moments.  Our results revolve around the dimensions and defining polynomials of these varieties.
\end{abstract}

\ccsdesc[500]{Mathematics of computing~Nonparametric statistics}
\ccsdesc[500]{Mathematics of computing~Probabilistic representations}
\keywords{nonparametric statistics, conditional independence, method of moments, varieties, toric, secant, dimension, finiteness
}



\maketitle
\section{Introduction}

Consider $n$ independent random variables $X_1,X_2,\ldots,X_n$ on the line $\mathbb{R}$.
We make no assumptions about the $X_k$ other than that their moments
 $\mu_{ki} = \mathbb{E}(X_k^i)$ exist. Then, by
\cite[Theorem 30.1]{Bil}, the random variable $X_k$ is uniquely characterized by its
sequence of
moments $\mu_{k1}, \mu_{k2}, \mu_{k3}, \ldots$.
 These moments 
   satisfy the Hamburger moment condition, which states that
 \begin{equation}
 \label{eq:hamburger} \begin{small}
\begin{bmatrix}
\mu_{k0} & \mu_{k1} & \mu_{k2} & \ldots \\
\mu_{k1} & \mu_{k2} & \mu_{k3} & \ldots \\
\mu_{k2} & \mu_{k3} & \mu_{k4} & \ldots \\
\vdots & \vdots & \vdots & \ddots 
\end{bmatrix} \end{small} \text{is positive semi-definite for all $k$.}
\end{equation}
For us, the $\mu_{ki}$ are unknowns. The only equations we require are
$\mu_{k 0} = 1$ for $k=1,2\ldots,n$.

 We write $m_{i_1 i_2 \cdots i_n}$ for the moments of 
 the random vector $X = (X_1, X_2,\ldots,X_n)$.
 The moments are the expected values of  the monomials
 $X_1^{i_1} X_2^{i_2} \cdots X_n^{i_n}$.
  By independence, we~have
     \begin{equation}
\label{eq:param}
\begin{split}
    m_{i_1 i_2 \cdots i_n}  &= 
 \mathbb{E}(X_1^{i_1} X_2^{i_2} \cdots X_n^{i_n})  \\
 &= 
 \mathbb{E}(X_1^{i_1} ) \mathbb{E}(X_2^{i_2})  \cdots 
 \mathbb{E}(X_n^{i_n}) =
   \mu_{1 i_1} \mu_{2 i_2} \cdots \mu_{n i_n} .
\end{split}
\end{equation} 
We examine this squarefree monomial parametrization for 
the moments with $i_1 + i_2 + \cdots + i_n = d$.
Its image is a toric variety $\mathcal{M}_{n,d}$ 
in the  projective space $\mathbb{P}^{\binom{n+d-1}{d}-1}$ 
of symmetric tensors.

\begin{example}[$n=d=3$] \label{ex:dreidrei}
The moment variety $\mathcal{M}_{3,3}$
 is defined by the monomial parametrization
 $m_{i_1 i_2 i_3} = \mu_{1 i_1} \mu_{2 i_2} \mu_{3 i_3}$ for $i_1 + i_2 + i_3 = 3$.
 We find that $\mathcal{M}_{3,3}$ is a cubic hypersurface
in the space $\mathbb{P}^9$ of symmetric $3 \times 3 \times 3$ tensors.
It is defined by
$\,m_{012} m_{120} m_{201} =  m_{021} m_{210} m_{102} $.
\end{example} 

In this article we study
mixtures of $r$ independent distributions.
 The associated moment variety  $\sigma_r(\mathcal{M}_{n,d}) $  is the $r$th secant variety
of the toric variety $\mathcal{M}_{n,d}$. It is parametrized~by
 \begin{equation} 
 \label{eq:paramm}
 \begin{split}
  m_{i_1 i_2 \cdots i_n} \,\, &= \,\,\sum_{j=1}^r \mu_{1 i_1}^{(j)} \,\mu_{2 i_2}^{(j)} \,\cdots \,\mu_{n i_n}^{(j)} \\
  & \!\!\!\! \hbox{where $i_1,i_2,\ldots,i_n \geq 0$ and $ i_1+i_2+\cdots+i_n=d$.}
\end{split}
  \end{equation}
   These are the moment varieties 
   in our title. Mixture weights can be ommitted in
   (\ref{eq:paramm}) since we work in projective geometry.

   We study these and their images under certain
   coordinate projections
$\mathbb{P}^{\binom{n+d-1}{d}-1} \dashrightarrow \mathbb{P}^{|N_\lambda|-1}$.
  Here $\lambda$  is any partition of $d$, and $N_\lambda$ is the set 
  of moments
 $m_{i_1 i_2 \ldots i_n}$ where $\{i_1,i_2,\ldots,i_n\}\backslash \{0\}$
 equals $\lambda$ as a multiset.
The images in $\mathbb{P}^{|N_\lambda|-1}$ of
the restricted parametrizations  (\ref{eq:param}) and (\ref{eq:paramm}) 
are denoted by $\mathcal{M}_{n,\lambda}$ and $\sigma_r(\mathcal{M}_{n,\lambda})$.
The restricted varieties make sense for statistics because
they refer to subclasses of moments that are natural
when infering parameters.
We note that
$\mathcal{M}_{n,\lambda}$
is also a toric variety and $\sigma_r(\mathcal{M}_{n,\lambda})$
is its $r$th secant variety.\

\begin{example}[$n=5,d=3$] \label{ex:zwei}
There are three partitions~$\lambda = (111)$,
$\lambda = (21)$ and $\lambda = (3)$. 
The number of moments $m_{i_1 i_2 i_3 i_4 i_5}$ equals
$\binom{5+3-1}{3} = 35$. This is the sum of
$N_{(111)} = 10$,  $N_{(21)} = 20$ and $N_{(3)} = 5$.
The following three toric varieties have dimensions $4,8,4$ respectively:
$$ \begin{matrix} \mathcal{M}_{5,(111)} \subset \PP^9  : m_{11100} = \mu_{11} \mu_{21} \mu_{31} ,\,  \\
m_{11010} = \mu_{11} \mu_{21} \mu_{41}, \,
\ldots,\, m_{00111} = \mu_{31} \mu_{41} \mu_{51} . \\
\,  \mathcal{M}_{5,(21)} \, \subset \PP^{19}  :  m_{21000} = \mu_{12} \mu_{21} ,\, 
m_{12000} = \mu_{11} \mu_{22},\, \\
m_{20100} = \mu_{12} \mu_{31} ,\, \ldots, \,
m_{00012} = \mu_{41} \mu_{52}. \\
\mathcal{M}_{5,(3)}  \,= \PP^4  :  m_{30000} = \mu_{13}, \,m_{03000} = \mu_{23} ,\,\ldots,\,
m_{00003} = \mu_{53}.\end{matrix} $$
Combining these parametrizations yields the  $14$-dimensional moment variety
$\mathcal{M}_{5,3} \subset \mathbb{P}^{34}$.
We will discuss the ideals of these toric varieties
and their secant varieties later on.
\end{example}

Our study is a sequel to the  work of Zhang and Kileel in \cite{ZK}.
That article takes an applied data science perspective and it offers
numerical algorithms for learning the parameters $\mu^{(j)}_{ki}$
from empirical moments $m_{i_1 i_2 \cdots i_n}$.
The primary focus in \cite{ZK} lies on numerical 
tensor methods for this recovery task. 
A key ingredient for their approach is identifiability,
which means that the dimension of the moment
variety $\sigma_r( \mathcal{M}_{n,\bullet})$
matches the number of free parameters.

The present paper lies at the interface of computer algebra
and nonparametric statistics. Our set-up is nonparametric
in the sense that no model assumptions are made
on the  constituent random variables on $\mathbb{R}$.
Conditional independence arises by passing via (\ref{eq:paramm}) to multivariate
 distributions on $\mathbb{R}^n$. This imposes semialgebraic constraints on
 the moments $m_{i_1 i_2 \cdots i_n}$.
 We disregard the inequalities in
 (\ref{eq:hamburger}) and focus on polynomial equations. This leads us to projective varieties, as is customary in algebraic statistics \cite{Sul}. Their defining equations provide test statistics for mixtures of products \cite[Section~3]{DSS}.

Our presentation is organized as follows. In Section \ref{sec2}
we demonstrate the wide range of interesting models
that are featured here.  Our scope includes themes from the early days of algebraic statistics:
 factor analysis~\cite{DSS} and
permutation data \cite[Section~6]{DS}.
 For the special partition $\lambda = (1,1,\ldots,1) $ we obtain the
 toric ideals associated with hypersimplices. 
 
 In Section \ref{sec3} we show that our varieties
exhibit finiteness up to symmetry, in the sense of
Draisma and collaborators \cite{BD, Draisma, DEFM, DEKL}.
Namely, if $r,d,\lambda$ are fixed and $n$  is unbounded then
finitely many $S_n$-orbits of polynomials suffice to cut out the varieties 
$\sigma_r( \mathcal{M}_{n,\bullet})$. Therefore,
computer algebra can be useful for 
high-dimensional data analysis in the setting of~\cite{ZK}.

Section \ref{sec4} is a detailed study of the
toric varieties  $\mathcal{M}_{n,\bullet}$.
We determine their dimensions, and we
investigate their polytopes and toric ideals. 
The ideal for $\mathcal{M}_{n,\lambda}$ is generated by quadrics and cubics, but the ideal for $\mathcal{M}_{n,d}$ is more complicated.
In Section \ref{sec5} we turn to identifiability of the secant varieties
$\sigma_r( \mathcal{M}_{n,\bullet})$  for $ r \geq 2$.
We present what we know about their dimensions.
Our main results are Theorems \ref{thm:integerprogram} and
\ref{thm:tropical}. These rest on 
integer programming and tropical geometry.

The parametrization (\ref{eq:paramm}) represents
a challenging implicitization problem.
In Section \ref{sec6} we report on some
computational results, featuring both symbolic
and numerical methods.  
For most examples in this paper we
used symbolic computations.

\section{Familiar Varieties}\label{sec2}
The study of highly structured projective varieties is
a main theme in algebraic statistics. 
This includes varieties of discrete probability distributions as well as
moment varieties of continuous distributions; 
see e.g.~\cite{AFS, KSS}. Note that
Veronese varieties fall into both categories.
In this section we match 
some of our moment varieties $\sigma_r( \mathcal{M}_{n,\bullet})$
with the existing literature.

We begin with $d=2$, so each $(i_1,i_2,\ldots,i_n)$ has
at most two non-zero entries. We change notation so that
the second moments are the entries of the
$n \times n$ covariance matrix $(m_{ij})$.

\begin{example}[$n=5,d=2$] The parametrization 
of the toric variety $\mathcal{M}_{5,2}$ 
given in (\ref{eq:param}) 
is written in matrix form~as 
\begin{align*}
&\qquad\begin{small}
\begin{bmatrix}
m_{11} & m_{12} & m_{13} & m_{14} & m_{15} \\
m_{12} & m_{22} & m_{23} & m_{24} & m_{25} \\
m_{13} & m_{23} & m_{33} & m_{34} & m_{35} \\
m_{14} & m_{24} & m_{34} & m_{44} & m_{45} \\
m_{15} & m_{25} & m_{35} & m_{45} & m_{55} 
\end{bmatrix} 
\end{small}
= \\
&
\begin{small}
\begin{bmatrix}
\mu_{12}  & \mu_{11} \mu_{21} &  \mu_{11} \mu_{31} & \mu_{11} \mu_{41} & \mu_{11} \mu_{51} \\
\mu_{11} \mu_{21}  & \mu_{22} &  \mu_{21} \mu_{31} & \mu_{21} \mu_{41} & \mu_{21} \mu_{51} \\
\mu_{11} \mu_{31}  & \mu_{21} \mu_{31} &  \mu_{32} & \mu_{31} \mu_{41} & \mu_{31} \mu_{51} \\
\mu_{11} \mu_{41}  & \mu_{21} \mu_{41} &  \mu_{31} \mu_{41} & \mu_{42}  & \mu_{41} \mu_{51} \\
\mu_{11} \mu_{51}  & \mu_{21} \mu_{51} &  \mu_{31} \mu_{51} & \mu_{41} \mu_{51} &  \mu_{52} 
\end{bmatrix}. 
\end{small}
\end{align*}

This shows that $\mathcal{M}_{5,2}$ is the join in $\mathbb{P}^{14}$ of the projective space 
$\mathcal{M}_{5,(2)} = \mathbb{P}^4 $  with coordinates $m_{ii} = \mu_{i2}$ and
the $4$-dimensional toric variety $\mathcal{M}_{5,(11)} \subset \mathbb{P}^9$
given by the off-diagonal entries.
The toric fourfold $\mathcal{M}_{5,(11)}$ has degree $11$. This is visualized in 
\cite[Figure 9-2]{GBCP}. Its ideal is generated by ten binomial quadrics, like $m_{12} m_{34} - m_{13} m_{24}$, as shown in 
\cite[Equation (9.2)]{GBCP}.

Passing from $r=1$ to $r=2$, we note that the secant variety  has the join decomposition
\begin{align*}
\sigma_2( \mathcal{M}_{5,2}) \, = \,
\sigma_2 \bigl( \,\mathbb{P}^4 \star \mathcal{M}_{5,(11)} \bigr) \,
&= \,
\mathbb{P}^4 \star \sigma_2 ( \mathcal{M}_{5,(11)} )  \,\, \\
&\subset \,\,
\mathbb{P}^4 \star \mathbb{P}^9 \, = \, \mathbb{P}^{14}. 
\end{align*}
The star denotes the join of projective varieties, which arises
from the Minkowski sum of the corresponding affine cones.

The prime ideal of the model $\sigma_2( \mathcal{M}_{5,2})$
is found by eliminating the diagonal entries $m_{ii}$
from the ideal of $3 \times 3$-minors of 
the symmetric $5  \times 5$ matrix $( m_{ij} )$.
The elimination ideal is principal, and its generator
is the polynomial
\begin{equation}
\label{eq:pentad} \!\!\!\!\! \begin{small} \begin{matrix} 
\,\,\,\, m_{12} m_{13} m_{24} m_{35} m_{45}
{-}m_{12} m_{13} m_{25} m_{34} m_{45}
{-}m_{12} m_{14} m_{23} m_{35} m_{45} \\
+ m_{12} m_{14} m_{25} m_{34} m_{35} 
{+}m_{12} m_{15} m_{23} m_{34} m_{45}
{-}m_{12} m_{15} m_{24} m_{34} m_{35} \\
+m_{13} m_{14} m_{23} m_{25} m_{45}
{-}m_{13} m_{14} m_{24} m_{25} m_{35} 
{-}m_{13} m_{15} m_{23} m_{24} m_{45}\\
+m_{13} m_{15} m_{24} m_{25} m_{34}
{+}m_{14} m_{15} m_{23} m_{24} m_{35}
{-}m_{14} m_{15} m_{23} m_{25} m_{34}.
\end{matrix} \end{small}
\end{equation}
This quintic is known as the {\em pentad}, and it
plays an important role in {\em factor analysis} \cite{DSS}.
In conclusion,  the $8$-dimensional moment variety 
$\sigma_2(\mathcal{M}_{5,(11)})$ is already familiar to statisticians.
\end{example}

We now generalize to the toric variety defined by
$\lambda = (1^d) = (1,1,\ldots,1)$ for any $n > d$.

\begin{remark} \label{rmk:eulerian}
The moment variety $\mathcal{M}_{n,(1^d)}$ is the toric variety
 associated with the hypersimplex
$$ \Delta(n,d) = {\rm conv} \bigl\{
e_{l_1} + e_{l_2} + \cdots + e_{l_d} \,:\,
1 \leq l_1 < l_2 < \cdots < l_d \leq n \bigr\}. $$
The variety  $\mathcal{M}_{n,(1^d)}$ has dimension
 $n-1$, it lives in $\mathbb{P}^{\binom{n}{d}-1}$
and its degree is the Eulerian number $A(n,d)$.
This is the number of permutations of $[n] = \{1,2,\ldots,n\}$
which have exactly $d$ descents.
Indeed, the degree of any projective toric variety equals
the normalized volume of the associated polytope \cite[Theorem 4.16]{GBCP},
and the formula ${\rm Vol}(\Delta(n,d)) = A(n,d)$
 is well-known in algebraic combinatorics. In \cite[Theorem 2.2]{Liu} it
 is attributed to Laplace.
\end{remark}

In the case of the second hypersimplex, one
can compute the prime ideal of
$\sigma_r(\mathcal{M}_{n,(11)})$
by eliminating the diagonal entries $m_{ii}$
from the ideal of $(r{+}1) \times (r{+}1)$ minors
of the covariance matrix $(m_{ij})$. 
This elimination problem is tough. For some instances see
 \cite[Table~1]{DSS}.

\begin{example} The moment variety 
$\sigma_5(\mathcal{M}_{9,(11)})$
is a hypersurface of degree $54$ in  $\mathbb{P}^{35}$.
The representation of its equation by means of resultants is explained in
\cite[Example~24]{DSS}.
\end{example}

We now turn to a scenario that  played a 
pivotal role in launching algebraic statistics in
the 1990s, namely the
spectral analysis of permutation data,
as described in \cite[Section 6]{DS}.
This is based on the toric ideal associated with the
Birkhoff polytope, whose vertices are the
$n!$ permutation matrices of size $n \times n$.
For an algebraic discussion see \cite[Section 14.B]{GBCP}.

\begin{proposition}\label{prop:birkhoff}
The moment variety $\mathcal{M}_{n,\lambda}$
for the partition $\lambda = (n-1,n-2,\ldots,2,1)$
is the toric variety of the Birkhoff polytope, which lives
in $\mathbb{P}^{n!-1}$ and has dimension $(n-1)^2$.
\end{proposition}

\begin{proof}
The moment coordinates $m_{i_1 i_2 \ldots i_n}$
for $\mathcal{M}_{n,\lambda}$ are
indexed by the $n!$ permutations of
$\{0,1,2,\ldots,n-1\}$.
The monomials on the right hand side of (\ref{eq:param})
have degree $n-1$ in  $n(n-1)$ distinct parameters $\mu_{ki}$.
The exponent vectors of these monomials can be identified with the
permutation matrices from which the first row has been removed.
This removal is an affine isomorphism, so it preserves the
Birkhoff polytope, which has dimension $(n-1)^2$.
\end{proof}

\begin{example}[$n=4$] \label{ex:birkhoff4}
The toric variety $\mathcal{M}_{4,(321)} \subset \mathbb{P}^{23}$
has dimension $9$ and degree $352$.
Its ideal is minimally generated by
$18$ quadrics like $m_{0123} m_{1032} -  m_{0132} m_{1023}$
and $160$ cubics~like $\,m_{0123} m_{1203} m_{2013} -  m_{0213} m_{2103} m_{1023} $.
The latter are induced from the $n=3$ case in Example~\ref{ex:dreidrei}.
 See \cite[Example 14.7]{GBCP} for a Gr\"obner basis
 and \cite[Section 6.1]{DS} for a statistical perspective.
\end{example}

\begin{example}[$n=5$] \label{ex:birkhoff5}
This appears in the last paragraph of \cite[Section 6.1]{DS}.
The toric variety $\mathcal{M}_{5,(4321)} \subset \mathbb{P}^{119}$
has dimension $16$ and degree $4718075$.
Its ideal is minimally generated by $1050$ quadrics and
$28840$ cubics. 
\end{example}

Yamaguchi, Ogawa and Takemura \cite{YOT} showed that the
toric ideal for the variety in Proposition \ref{prop:birkhoff} is always
generated in degree two and three.
Theorem \ref{thm:generated-cubics} generalizes this~result.

The degrees reported in Examples \ref{ex:birkhoff4}
and \ref{ex:birkhoff5} are the
volume of the Birkhoff polytope.
This volume is known up to $n=10$ \cite{BP, BP-arxiv}.

\section{Finiteness}
\label{sec3}

We consider
the projective variety $\sigma_r(\mathcal{M}_{n,d}) \subset \mathbb{P}^{\binom{n+d-1}{d}-1}$ 
defined by   (\ref{eq:paramm}).
 This   parametrization can be understood as follows without any reference to probability or statistics.
 Namely, $m_{i_1 i_2 \cdots i_n}$ is the coefficient
 of the monomial $x_1^{i_1} x_2^{i_2} \cdots x_n^{i_n}$ in the expansion of
 the polynomial
 \begin{equation}
 \label{eq:bigM}
  F(x_1,x_2, \ldots,x_n) = \sum_{j=1}^r \,f^{(j)}_1(x_1) \, f^{(j)}_2(x_2)  \,\,\cdots\,\, f^{(j)}_n(x_n),
    \end{equation}
 where $f^{(j)}_k(x) = 1 + \mu^{(j)}_{k1} x +  \mu^{(j)}_{k2} x^2 + \cdots + 
  \mu^{(j)}_{kd} x^d$ are $rn$ unknown univariate polynomials.
   For any partition $\lambda$ of $d$, let $N_\lambda$ be the subset  of coefficients
 $m_{i_1 i_2 \ldots i_n}$ where $\{i_1,i_2,\ldots,i_n\} \backslash \{0\}$
 equals $\lambda$ as a multiset.
  The variety $\sigma_r(\mathcal{M}_{n,\lambda})$  is the closure of
the image of $\sigma_r(\mathcal{M}_{n,d})$ under the map
$\mathbb{P}^{\binom{n+d-1}{d}-1} \dashrightarrow \mathbb{P}^{|N_\lambda|-1}$.
The polynomial in (\ref{eq:bigM}) is the
truncated moment generating function.
Taking $r=1$  we obtain the
toric varieties $\mathcal{M}_{n,\bullet}$.

The equations for these varieties satisfy {\em finiteness up to symmetry} when
$r, d,\lambda$ are fixed and $n$ grows.
Here symmetry refers to the
action of the symmetric group $S_n$ on our varieties,
their parametrizations (\ref{eq:paramm}), and their prime ideals.
These ideals satisfy natural inclusions
$$ \qquad I(\sigma_r(\mathcal{M}_{n, \bullet}))\, \subset \, I(\sigma_r(\mathcal{M}_{n+1,\bullet})) ,\qquad {\rm where} \quad
\bullet \in \{d,\lambda\},$$
by appending a zero to the indices of every coordinate.
In symbols, $\,m_{i_1 i_2 \cdots i_n} \,\mapsto \, m_{i_1 i_2 \cdots i_n 0 }$.
If we iterate these inclusions and let the big symmetric group act, then we obtain inclusions
\begin{equation}
\label{eq:inclusions} \langle S_n I(\sigma_r(\mathcal{M}_{n_0,\bullet})) \rangle \,\subseteq \,
I(\sigma_r(\mathcal{M}_{n, \bullet})) 
\quad \,{\rm for} \,\, n > n_0.  \quad
\end{equation}
{\em Ideal-theoretic finiteness} means that there exists
  $n_0$ such that  equality holds for all $n > n_0$.
The weaker notion of {\em set-theoretic finiteness} means that  equality holds in (\ref{eq:inclusions}) 
after the left ideal is enlarged to its radical. The smallest possible $n_0$, if it exists, is a
 function of $r,d,\lambda$.

\begin{example}[Pentads and $3 \times 3$ minors] \label{ex:pentadlimit}
If $r=2$ and $\lambda = (11)$ then
  ideal-theoretic
finiteness holds with $n_0 = 6$. This was proved by Brouwer and Draisma in \cite[Theorem~1.7]{BD}
in response to \cite[Conjecture 26]{DSS}.
The prime ideal of $\sigma_2( \mathcal{M}_{n,(11)})$
is generated 
 by the $\binom{n}{5}$ pentads and
the $5 \binom{n}{6} $ off-diagonal $3 \times 3$ minors of a symmetric $n \times n$ matrix.
See \cite[Section~3]{BD} for an equivariant Gr\"obner basis.
If $r=1$ then
 ideal-theoretic finiteness holds with $n_0 = 4$, by
 the Gr\"obner basis in
 \cite[Theorem 9.1]{GBCP} for
 the toric ideal of the second~hypersimplex.
\end{example}

In recent years, there has been considerable progress on
commutative algebra in
infinite polynomial rings with an action of the infinite symmetric group,
or of rings over the category FI of finite sets with injections. 
The following result reflects the state of the art on that topic.

\begin{theorem}
Given any partition $\lambda \vdash d$ and integer $r \geq 1$,
set-theoretic finiteness holds for the varieties
 $\sigma_r(\mathcal{M}_{n,d})$
and $\sigma_r(\mathcal{M}_{n,\lambda})$.
Ideal-theoretic finiteness holds in the toric case $r=1$.
\end{theorem}

\begin{proof}
The statement about toric varieties $(r=1)$ follows  from
\cite[Theorem 1.1]{DEKL}. For $r \geq 2$, we apply the main theorem
in 
\cite{DEFM}. 
First we consider the varieties $\sigma_r(\mathcal{M}_{n,d})$.  Use the map (\ref{eq:paramm}) 
between two polynomial rings in countably many variables.  
This is a morphism of FI-algebras as in \cite[Section 1.1]{DEFM}. 
In the formulation of \cite[Corollary 1.1.2]{DEFM}, the parametrization 
takes the $rd \times \mathbb{N}$ matrix whose entries are
$\mu^{(j)}_{ki}$ to 
the $\mathbb{N} \times \cdots \times \mathbb{N}$ ($d$ times) tensor whose entries are 
$m_{i_1 \cdots i_n}$ viewed as degree-$d$ moments in $n$ dimensions. 
The closure of the image
is topologically ${\rm Sym}(\mathbb{N})$-Noetherian, which yields 
 set-theoretic finiteness for $\sigma_r(\mathcal{M}_{n,d})$.  
 The case of $\sigma_r(\mathcal{M}_{n,\lambda})$ is similar. 
\end{proof}

\begin{remark}
It is conjectured in \cite[Conjecture 1.1.3]{DEFM} that the main result
in \cite{DEFM}  holds ideal-theoretically. 
This would imply ideal-theoretic finiteness for
 $\sigma_r(\mathcal{M}_{n,d})$
and $\sigma_r(\mathcal{M}_{n,\lambda})$.
\end{remark}

Whenever ideal-theoretic finiteness holds,
one can try to use equivariant Gr\"obner bases~\cite{BD}
for computing the desired finite generating set.
An implementation for  the toric case is~described  in \cite{DEKL}, but
we found this to be quite slow. 
The case $\lambda = (11)$ 
is covered by Example~\ref{ex:pentadlimit}.

\begin{example}[Cycles in bipartite graphs] \label{ex:cycles} If $r=1$ and $\lambda = (21)$ then
ideal-theoretic finiteness holds with $n_0 = 4$.
Namely, the toric ideal of $\mathcal{M}_{n,(21)}$ is generated
by $6\binom{n}{4}$ quadrics and $\binom{n}{3}$ cubics. 
This follows from \cite[Lemma~1.1]{OH}. 
Indeed, the above
binomials correspond to the chordless cycles in the bipartite
graph that is obtained from $K_{n,n}$ by removing the
$n$ edges $(1,1),(2,2),\ldots,(n,n)$. For $n > 4$, every such 
chordless cycle is supported on
a bipartite subgraph of the same kind with $n_0 = 4$. 
\end{example}

\begin{example}[Hypersimplex]\label{ex:hypersimplex-finiteness} As seen in Section \ref{sec2}, when $r=1$ and $\lambda=(1^d)$, the moment variety $\mathcal{M}_{n,\lambda}$ is the toric variety associated to the hypersimplex $\Delta(n,d)$. Its ideal is generated by quadrics \cite[Section 14A]{GBCP}. The indices 
occurring in each quadratic binomial are $1$ in at most $2d$ of the $n$ coordinates.
Therefore, ideal-theoretic finiteness holds with $n_0=2d$.
\end{example}

We close with a corollary that generalizes the previous two examples.
Its proof rests on a forward reference to the next section, where we derive
various results for our toric ideals.

\begin{corollary}    
Fix a partition $\lambda$ with $e$ nonzero parts, fix $\,r=1$, and suppose that $n$ increases.   The
toric varieties $\mathcal{M}_{n,\lambda}$ satisfy
ideal-theoretic finiteness for some $\,n_0 \leq 3e$ where $e$ is the length of $\lambda$. 
\end{corollary}

\begin{proof}
Theorem~\ref{thm:generated-cubics} says that the ideal of
$\mathcal{M}_{n,\lambda}$ is generated by binomials of degree at most~$3$.
Each of the two monomials in such a binomial is a product of two or three
variables $m_{i_1 i_2 \cdots i_n}$. The two monomials have the
same $A$-degree, where $A$ is the matrix representing (\ref{eq:param}). This
implies that the
slots $\ell \in \{1,2,\ldots,n\}$ where a nonzero index
$i_\ell$ occurs are the same in both monomials.
The total number of such slots is at most $3e$.
This yields the bound $n_0 \leq 3e$.  
\end{proof}

\begin{remark}
Every partition $\lambda \vdash d$ satisfies $e \leq d$, and
equality holds only for $ \lambda = (1^d)$, as in
Example \ref{ex:hypersimplex-finiteness}.
 For $\lambda = (21)$ in Example \ref{ex:cycles}, we have $e=2$,
and this yields $n_0 = 2e= 4$.
\end{remark}

\section{Toric Combinatorics}
\label{sec4}

In this section we study $\mathcal{M}_{n,d}$ and $\mathcal{M}_{n,\lambda}$ for some partition $\lambda$ of $d$. 
With each such toric variety we associate a 0-1 matrix $A$ as in \cite{GBCP}
whose columns correspond to the monomials in (\ref{eq:param}).
The rank of $A$ is one more than the dimension of the projective toric variety.
We first show that  $\mathcal{M}_{n,d}$ has the expected dimension, namely
 the number of parameters minus one.

\begin{theorem}\label{thm:expected-dim}
    The dimension of the moment variety $\mathcal{M}_{n,d}$ is $$\min\left\{nd-1,\binom{n+d-1}{d}-1\right\}.$$
\end{theorem}

\begin{proof}
First assume $n > d$. We will show that the $A$-matrix associated to the moment variety has rank $nd$ by displaying a nonzero $nd\times nd$ minor. Consider the $d$ special partitions
\begin{equation}
\label{eq:columnblocks}
(1,1,\ldots,1),\; (2,1,\ldots,1),\ldots, (d-2,\; 1,1),\; (d-1,1),\; (d).
\end{equation}
Each of these partitions induces (by permutation) at least $n$ columns in the $A$-matrix. For each $(k,1,\ldots,1)\vdash d$, pick $n$ of these columns such that $k$ appears in each of the $n$ spots. The principal submatrix of $A$ induced by all these columns is an $nd\times nd$ matrix of the form
$$
B\,\,=\,\, \begin{small} \left[ 
\begin{array}{c|c|c|c|c|c} 
  M & * & * & \ldots & * & 0\\ 
  \hline 
  0 & I_n &  0& \ldots & 0 & 0\\
  \hline
 0& 0 & I_n & \ldots & 0 & 0\\
 \hline
 0& 0 & 0 & \ddots & 0 & 0\\
 \hline
 0& 0 & 0 & \ldots & I_n & 0\\
 \hline
 0& 0 & 0 & \ldots & 0 & I_n\\
\end{array} 
\right] , \end{small}
$$
where the $d$ row blocks are labeled $(\mu_{k1}: k\in[n]),\ldots,(\mu_{kd}: k\in[n])$  and the $d$ column blocks are (\ref{eq:columnblocks}). The matrix $M$ gives
a column basis for the $A$-matrix of the hypersimplex variety $\mathcal{M}_{n,(1,1,\ldots,1 )}$, so 
it is invertible.
We conclude $\det B=\det M\neq 0$, and so $\text{rank}(A)=nd$.

Now suppose $n \leq d$. Index the columns of the $A$-matrix by permutations of $(i_1,\ldots, i_n)$ with $i_1+\cdots+i_n=d$ ordered reverse-lexicographically. Index the rows by
$\mu_{11}, \mu_{12},\ldots,\mu_{nd}$. The principal submatrix on the first $2d+1$ rows and columns 
is invertible, so the first $2d+1$ columns of $A$ are linearly independent. From the remaining columns, we pick $d(n-2)-1$ of them such that for every $j=2d+1,\ldots,nd$ exactly one has 1 in the $j$th coordinate. In this way we obtain $nd$ linearly independent columns of $A$. Therefore, $A$ has full rank.
\end{proof}

Given a partition $\lambda \vdash d $ padded by zeroes to have length $n$, we define a partition $\nu$,
called the {\em reduction} of $\lambda$.
Let $k_0 \geq \cdots  \geq k_s$ be the multiplicities
of the distinct parts  in $\lambda$.
Then 
\begin{equation}
\label{eq:reducedpartition}
   \nu\,\,=\,\,\bigl(\,\underbrace{s,\ldots,s}_{k_s},\,\underbrace{s-1,\ldots,s-1}_{k_{s-1}},\,\ldots\,,\,\underbrace{1,\ldots,1}_{k_1},\,\underbrace{0\ldots,0}_{k_0}\, \bigr).
\end{equation}   
We write $s$ for the largest part of $\nu$, so
$s+1$ is the number of distinct parts of $\lambda$.
For example, the partitions $(8,5,5,4)$ and $ (7,7,3,0)$ have
the  same reduction $\nu=(2,1,0,0)$, with $s=2$.

\begin{lemma}\label{lem:partition-identification}
If $\nu$ is the reduction of $\lambda$ then $|N_\lambda| = |N_\nu|$ and
     $\mathcal{M}_{n,\lambda}=\mathcal{M}_{n,\nu}$ in $\mathbb{P}^{|N_\nu| - 1}$.
\end{lemma}

\begin{proof}
Let the $\mu_{k0}$ be unknowns in the monomial parametrization (\ref{eq:param}). The image of this
altered map also equals $\mathcal{M}_{n,\lambda}$.
The toric variety $\mathcal{M}_{n,\nu}$ has the same parametrization, after changing
the index $i \in \lambda$ in each parameter $\mu_{ki}$ to the corresponding entry in~$\nu$.
\end{proof}

\begin{example}[Hypersimplex]
If $s=1$ and $\lambda = (1^d)$ with $n/2 < d < n$ then
$\nu = (1^{n-d})$ in Lemma~\ref{lem:partition-identification}, and
we recover the identification of the hypersimplices $\Delta(n,d) $ and $ \Delta(n,n-d)$.
\end{example}

\begin{theorem}\label{thm:dim-partitions}
        The moment variety $\mathcal{M}_{n,\lambda} = \mathcal{M}_{n,\nu}$ has dimension $(n-1)s$,
        for  $\nu $ in \textup{(\ref{eq:reducedpartition})}.
\end{theorem}

\begin{proof}
    We must show that the $A$-matrix of $\mathcal{M}_{n,\nu}$ has rank $(n-1)s+1$. We proceed by induction on $s$, the base case $s=1$ being the hypersimplex.
     We partition the rows of $A$ into $s$ blocks $(\mu_{k1}: k\in[n]),\ldots,(\mu_{ks}: k\in[n])$. 
     The rows of $A$ in the $i$th block
     sum to the constant vector $(k_i,k_i,\ldots k_i)$. Hence, the rank of $A$ is bounded above by $(n-1)s+1$. We will show that this is also a lower bound by displaying an invertible submatrix of this size.
    
    First, assume that $k_s>1$. Consider the columns of the $A$-matrix indexed by $m_{i_1\cdots i_n}$ such that $i_1=s$. By induction on $n$, these columns induce a submatrix of rank $(n-2)s+1$. Hence, we may pick $ns-2s+1$ linearly independent columns from this set. Next, for each $j=0,\ldots,s-1$, pick a column indexed by some $m_{i_1\cdots i_n}$ with $i_1=j$. By construction, adding these $s$ columns does not introduce dependence relations. We have constructed a set of $ns-2s+1+s=(n-1)s+1$ linearly independent columns of $A$, so $A$ has the desired~rank.

    Now consider the case when $k_s=1$. Again, consider the columns of the $A$-matrix indexed by $m_{i_1\cdots i_n}$ such that $i_1=s$. By induction,
    but now also on $s$, these columns induce a submatrix of rank $(n-2)(s-1)+1$. Next, add $n-1$ columns that are indexed by $m_{i_1\cdots i_n}$ where $i_1=s-1$ and such that for each $j=2,\ldots, n$ there is an index with $i_j=s$. Finally, add $s-1$ columns indexed by $m_{i_1\cdots i_n}$ such that for each $j=s-2,\ldots, 0$ there is an index with $i_1=j$ and $m_{i_2}\neq s, s-1$. This way we obtain $(n-2)(s-1)+1+(n-1)+(s-1)=(n-1)s+1$ columns, which are linearly independent by construction. Therefore, $A$ has the desired rank.
\end{proof}

The toric variety $\mathcal{M}_{n,d}$ is an aggregate of the
$\mathcal{M}_{n,\lambda} $ for $\lambda \vdash d$,
but there is no easy transition. For instance,
ideal generators for $\mathcal{M}_{n,d}$ do not restrict to
ideal generators for~$\mathcal{M}_{n,\lambda}$.

\begin{example}[$n=d=4$] \label{ex:M44toric}
The partitions $\lambda = (4),(31),(22)$, $(211),(1111)$
have the reductions $\nu=(1),(21),(11),(21),()$
with $s=1,2,2,2,0$. Two nontrivial varieties
$\mathcal{M}_{4,\nu}$ are given by the off-diag-onal entries of
$4 \times 4$-matrices. The variety $\mathcal{M}_{4,4} $
has dimension $15$ and degree $1072$ in $\mathbb{P}^{34}$, and its ideal
is generated by $52$ quadrics and $28$ cubics. 
The subset which involves the twelve unknowns
$m_{2110}, \ldots,$ $m_{0112}$ does not suffice to cut out
 $\mathcal{M}_{4,(211)}$ in $\mathbb{P}^{11}$. The ideal 
of  $\mathcal{M}_{4,(211)}$
 is generated by
$6$ quadrics and $4$ cubics, 
namely the cycles in Example~\ref{ex:cycles}.
\end{example}

The toric ideals for individual partitions are very nice.  Our next result builds upon \cite{YOT}.

\begin{theorem} \label{thm:generated-cubics}
For any partition $\lambda$, the ideal of ${\mathcal{M}_{n, \lambda}}$ is generated by quadrics and cubics.
\end{theorem}

\begin{proof}
For a partition $\lambda= (\lambda_1,\ldots,\lambda_{e})$, set $R = \mathbb{R}[m_{i_1 i_2 \cdots i_n} : \{i_1, \ldots, i_n\} \text{ is } \lambda]$. Let $I \subset R$ be the toric ideal defining $\mathcal{M}_{n, \lambda}$. Note that $\eta = (e, e-1, \ldots, 1)$ is a partition of the same length but possibly with a different sum.  Let $J \subset S$ be the toric ideal defining $\mathcal{M}_{n, \eta}$ where $S = \mathbb{R}[m_{j_1 j_2 \cdots j_n} : \{j_1, \ldots, j_n\} \text{ is } \eta]$.
Define a surjective ring homomorphism $\varphi : S \rightarrow R$ by mapping $m_{j_1 j_2 \cdots j_n}$ to $m_{i_1 i_2 \cdots i_n}$ where $\{i_1, \ldots, i_n\}$ is obtained from $\{j_1, \ldots, j_n\}$ by replacing $e$ with $\lambda_1$, replacing $e-1$ with $\lambda_2$, etc.
We claim that $\varphi(J) = I$.  
This implies the theorem because $J$ is generated by quadrics and cubics \cite[Theorem~2.1]{YOT} and $\varphi$ preserves $\mathbb{N}$-degree.

Since $I$ and $J$ are toric, we may verify $\varphi(J) = I$ on binomials.  
To show $\varphi(J) \subseteq I$, fix a binomial in $J$, say of degree $\delta$, written as
$\mathbf{b} = \prod_{\alpha=1}^\delta m_{j_{\alpha 1} \cdots j_{\alpha n}} - \prod_{\alpha=1}^\delta m_{j'_{\alpha 1} \cdots j'_{\alpha n}}$.
Encode this by two $\delta \times n$ matrices $B = (j_{\alpha \beta})$
and $C = (j'_{\alpha \beta})$.  
Membership in $J$ means that substituting the parametrization \eqref{eq:param} into $\mathbf{b}$
gives the result $0$, and
this is equivalent to the multiset of entries in corresponding columns of $B$ and $C$ being equal.
This property is preserved after replacing $e$ by $\lambda_1$, replacing $e-1$ by $\lambda_2$ etc. throughout $B$ and $C$.  So $\varphi(\mathbf{b}) \in I$ as desired.

To prove $I \subseteq \varphi(J)$, let $\mathbf{c} = \prod_{\alpha=1}^\delta m_{i_{\alpha 1} \cdots i_{\alpha n}} - \prod_{\alpha = 1}^\delta m_{i'_{\alpha 1} \cdots i'_{\alpha n}}$ be a binomial in $I$ encoded by matrices $D = (i_{\alpha \beta})$
and $E = (i'_{\alpha \beta})$.
We will construct a binomial $\mathbf{d} \in J$ such that $\varphi(\mathbf{d}) = \mathbf{c}$. 
In terms of matrices $D$ and $E$, in each of their rows we must choose one element that equals $\lambda_1$ and replace it by $e$, then choose another element that equals $\lambda_2$ and replace it by $e-1$, and so forth until the set of nonzero elements in each row has been replaced by $[e]$, in such a way so that the multiset of entries in corresponding columns of the transformed matrices $D$ and $E$ are equal.
To achieve this it suffices to consider distinct values in $\lambda$ one at a time.   

Without loss of generality, assume $\lambda = (1^e)$.
Now $D$ and $E$ have $e$ ones and $n-e$ zeros per each row.
To choose the elements to replace by $e$, we consider a bipartite multigraph between the rows of $D$ and the rows of $E$, where an edge is drawn between a row in $D$ and a row in $E$ for every column in which there is a $1$ in both rows.
A perfect matching would give a valid choice of elements to replace by $e$. 
Such a matching exists by Hall's Marriage Theorem.  Indeed, for any subset $W$ of rows in $D$ their neighborhood must contain at least $|W|$ rows in $E$. Otherwise, there exists a column in $D$ with more ones than the corresponding column in $E$, since each row contains the same number of ones.  But this contradicts $\mathbf{c} \in I$.
Similarly, we carry out the subsequent replacements.  
Thus a suitable binomial $\mathbf{d}$ exists.  
It follows $I \subseteq \varphi(J)$.  Combining with the preceding paragraph, we conclude $\varphi(J)=I$.
\end{proof}

By contrast, the ideals for $\mathcal{M}_{n,d}$ appear to be more complicated.  
We conjecture that there does not exist a uniform degree bound for their generators that is independent of $n, d$.

\begin{example}[$n{=}3, d{=}7$]
The toric variety $\mathcal{M}_{3,7}$ has dimension $20$ and degree $14922$ in $\mathbb{P}^{35}$.  Its ideal is minimally generated by $46$ cubics, $168$ quartics, $135$ quintics and $18$~sextics. 
\end{example}

\section{Secant Varieties}
\label{sec5}

Theorems  \ref{thm:expected-dim} and \ref{thm:dim-partitions}
gave the dimensions
of our moment varieties for $r=1$. 
We next focus on  $r \geq 2$, where $\sigma_r(\mathcal{M}_{n,d})$
and $\sigma_r(\mathcal{M}_{n,\lambda})$ are no longer toric.
We begin with an example.

\begin{example}[$n=5, d=3$] \label{ex:M53toric}
The toric variety $\mathcal{M}_{5,3}$ and its secant variety
$\sigma_2(\mathcal{M}_{5,3})$ live~in the projective space $\mathbb{P}^{34}$
of symmetric $5 \times 5 \times 5$ tensors;
see Example \ref{ex:zwei}.
By Theorem  \ref{thm:expected-dim},
we have ${\rm dim}(\mathcal{M}_{5,3}) = 14$.
The expected dimension of the secant variety $\sigma_2(\mathcal{M}_{5,3})$ would be $2 \cdot 14+1 = 29$.
However, we must subtract $5=4+1$ because
 $\mathcal{M}_{5,3}$ is a cone with apex 
$\sigma_2(\mathcal{M}_{5,(3)}) = \mathbb{P}^4$.
Therefore, $\sigma_2(\mathcal{M}_{5,3})$ has dimension $24$.
The prime ideal of $\sigma_2(\mathcal{M}_{5,3})$ will be
presented in Proposition \ref{prop:M53}. 
\end{example}

Our first result explains the drop in dimension seen in the example above.

\begin{proposition} \label{prop:expdim}
The dimension of the moment variety satisfies the upper bound
\begin{equation}
\label{eq:expdim} 
\begin{matrix} 
{\rm dim} \bigl(\sigma_r(\mathcal{M}_{n,d}) \bigr) \leq 
{\rm min} \bigl\{\,   rnd-rn+n-1,
 \,
\binom{n+d-1}{d}-1 \,\bigr\}.  
\end{matrix}
\end{equation}
\end{proposition}

\begin{proof}
The given toric variety is a cone over the
projective space $\mathcal{M}_{n,(d)} = \mathbb{P}^{n-1}$. In symbols,
$ \mathcal{M}_{n,d} =  \mathbb{P}^{n-1} \star  \widetilde{\mathcal{M}}_{n,d}, $
where $\widetilde{\mathcal{M}}_{n,d}$ is the toric variety given by
 all $\binom{n+d-1}{d}-n$ moments that involve more than one coordinate.
By counting parameters, we find ${\rm dim}(\widetilde{\mathcal{M}}_{n,d}) \leq n(d-1)-1$.
We obtain the  secant variety of the big toric variety as the join of the apex 
with the reduced toric variety:
$ \sigma_r(\mathcal{M}_{n,d})  =  \mathbb{P}^{n-1}  \star 
\sigma_r(\widetilde{\mathcal{M}}_{n,d}) .$
The dimension of the right-hand side is bounded above by
$ n + r \bigl( {\rm dim}(\widetilde{\mathcal{M}}_{n,d})
\bigr) + r-1 \, \leq \,
 n + r \cdot \bigl( n(d-1)-1\bigr) + r-1.  $
 This yields  (\ref{eq:expdim}).
\end{proof}

We found the inequality (\ref{eq:expdim}) to be strict when $r \geq n$.  The following sharper bound holds.  (To see it is sharper, consider $S = [d]$ and $S=\{d\}$ in \eqref{eq:ILP}.)

\begin{theorem} \label{thm:integerprogram}
The dimension of the secant variety $\sigma_r(\mathcal{M}_{n,d})$ is 
bounded above by the optimal value of the following integer linear programming problem:
\begin{equation}
\label{eq:ILP}
\begin{split}
&{\rm maximize} \,\,c_1 + c_2 + \cdots + c_d - 1  \\
&  \hbox{\rm subject to} \quad
\,\,0\,\, \leq \,c_i\, \leq \,nr \,\,\,\, \hbox{\rm for} \,\,\, i \in [d] \\ &  {\rm and} \,\,
\sum_{i \in S} c_i \,\leq \,\sum_{\lambda \cap S \neq \emptyset} |N_{\lambda}| \quad {\rm for} \,\,\, S \subseteq [d].
\end{split}
\end{equation}
The last sum ranges over partitions $\lambda \vdash d$ of length $\leq n$ having nonempty 
intersection with~$S$. 
\end{theorem}
\begin{proof}
The secant variety $\sigma_r(\mathcal{M}_{n,d})$ is parameterized by the polynomial map \eqref{eq:paramm}.  Therefore its dimension is one less than the maximal rank assumed by the differential of \eqref{eq:paramm}.
This Jacobian matrix has size $\binom{n+d-1}{d} \times nrd$, where the rows are labeled by  $m_{i_1 i_2 \cdots i_n}$ such that $i_1, \ldots ,i_n \geq 0$ and $i_1 + \cdots + i_n =d$, and the columns are labeled by  $\mu_{1 i}^{(j)},  \ldots, \mu_{n i}^{(j)}$ for $i=1, \ldots, d$ and $j=1, \ldots, r$.
We view this as a block matrix, where the rows are grouped according to the partition $\lambda$ given by $(i_1, \ldots, i_n)$ and the columns are grouped according to the degree $i$.  Notice that the matrix is sparse, in that a block labeled by $(\lambda, i)$ is nonzero only if $i \in \lambda$.  

Let $\mathcal{C}$ be a set of linearly independent columns in the Jacobian matrix, with $c_i$ columns labeled by $i$. 
The integers $c_i$ satisfy $0 \leq c_i \leq nr$ for $i = 1, \ldots, d$.  
Let $ S \subseteq [d]$ and $\mathcal{C}'$ the subset of columns in $ \mathcal{C}$ 
that are labeled by elements of $S$.
Since $\mathcal{C}'$ is linearly independent,
the number of rows which are nonzero in $\mathcal{C}$ exceeds $|\mathcal{C}'|$.
  By the aforementioned sparsity, 
\begin{equation}\label{eq:joe-bound}
\sum_{i \in S} c_i \quad \quad \leq \quad \sum_{\substack{\lambda \vdash d, \,\,\,  \lambda \cap S \neq \emptyset, \\[1pt] \textup{length of } \lambda \textup{ is at most } n}} \!\!\! |N_{\lambda}|.
\end{equation}
We conclude that $|\mathcal{C}|-1$ is bounded  above by the
maximum value in (\ref{eq:ILP}), as desired.
\end{proof}

Solving an integer linear program is expensive in general.  
However, the integer linear program  in (\ref{eq:ILP}) has a special structure which allows for
a greedy solution that is optimal.

\begin{theorem}\label{thm:greedy}
We construct a feasible solution for (\ref{eq:ILP}) greedily, starting with
 $\mathbf{c}^{(0)} = 0 $ in $\mathbb{Z}^d$. For $t = 1, \ldots, r$, choose $\mathbf{c}^{(t)} \in \mathbb{Z}^{d}$ such that $c_{i}^{(t-1)} \leq {c}_{i}^{(t)} \leq {c}_{i}^{(t-1)} + n$
  for all $i \in [d]$,   and if ${c}_{i}^{(t)} < {c}_{i}^{(t-1)} + n$ then there exists $S \subseteq [d]$ containing $i$ such that $\sum_{j \in S} c_{j}^{(t)} = \sum_{\lambda \cap S \neq \emptyset} |N_{\lambda}|$. 
   Then $\mathbf{c}^{(r)} \in \mathbb{Z}^d$  is optimal for the integer linear program (\ref{eq:ILP}).
\end{theorem}

\begin{proof}
We claim that $\mathbf{c}^{(r)}$ is optimal for the linear program (\ref{eq:ILP}), with integrality constraints dropped. 
The dual linear program has variables $y_S$ for $\emptyset \neq S \subseteq [d]$ and $z_i$ 
for $i \in [d]$. This dual linear program equals:
\vspace{-0.5em}
\begin{align*}
& {\rm minimize}\,
 \,\,\, nr(z_1 + \ldots + z_d)  \, +  \sum_{S \subseteq [d]} \Big{(} \sum_{\lambda \cap S \neq \emptyset} |N_{\lambda}| \Big{)} y_S \, - 1 \\[-0.5em]
 & \text{subject to } \mathbf{y} \in \mathbb{R}^{2^d - 1}_{\geq 0} , \mathbf{z} \in \mathbb{R}^{d}_{\geq 0}
  \, \,{\rm and} \,
z_i + \sum_{S \ni i} y_S \, \geq  1 \, {\rm for} \, \, i \in [d].
\end{align*}
It suffices to find a dual feasible point at which the dual objective equals the primal objective evaluated at $\mathbf{c}^{(r)}$.
We call a set $S \subseteq [d]$ \textit{saturated} if equality holds in \eqref{eq:joe-bound} for $\mathbf{c}^{(r)}$.  We define
\begin{align*}
& y_S^{(r)} = \begin{cases} 1 &\text{if } S \subseteq [d] \text{ is saturated and maximal such set} \\ 
0 & \text{otherwise}; \end{cases} \\[0.5em]
& z_i^{(r)} = \begin{cases} 
1 & \text{if } \nexists \, S \subseteq [d] \text{ s.t. } i \in S \text{ and } S \text{ is saturated}  \\
0 & \text{otherwise}.
\end{cases}
\end{align*}
The vector $(\mathbf{y}^{(r)}, \mathbf{z}^{(r)})$ is dual feasible.
   Further, we claim that there is a unique maximal saturated subset of $[d]$, possibly empty. 
    Suppose that $S$ and $T$ are saturated.  Then $\sum_{i \in T \setminus S} c_i^{(r)} + \sum_{i \in T \cap S} c_i^{(r)} = \sum_{\lambda \cap T \neq \emptyset} |N_{\lambda}|$ and $\sum_{i \in T \cap S} c_i^{(r)} \leq \sum_{\lambda \cap (T \cap S) \neq \emptyset} |N_{\lambda}|$ since $T$ is saturated and $\mathbf{c}^{(r)}$ is primal feasible.  Subtracting these, we find 
    \begin{align*}
    \sum_{i \in T \setminus S} c_i^{(r)} &\geq \sum_{\lambda \cap T \neq \emptyset} |N_{\lambda}| - \sum_{\lambda \cap (S \cap T) \neq \emptyset} |N_{\lambda}| \\
    &= \sum_{\lambda \cap T \neq \emptyset,  \lambda \cap (S \cap T) = \emptyset} |N_{\lambda}|.
    \end{align*} Adding $\sum_{i \in S} c_i^{(r)} = \sum_{\lambda \cap S \neq \emptyset} |N_{\lambda}|$ implies $$\sum_{i \in S \cup T} c_i^{(r)} \geq \sum_{\lambda \cap (S \cup T) \neq \emptyset} |N_{\lambda}|.$$ Hence $S \cup T$ is saturated by primal feasibility.  Thus there is a unique maximal saturated subset of $[d]$. 
It follows that the dual objective evaluated at $(\mathbf{y}^{(r)}, \mathbf{z}^{(r)})$ equals the primal objective evaluated at $\mathbf{c}^{(r)}$. This completes the proof.
\end{proof}

We conjecture that the integer linear program  (\ref{eq:ILP}) computes the correct dimension: 

\begin{conjecture}
If $d \geq 3$ then the bound for $\operatorname{dim}(\sigma_r(\mathcal{M}_{n,d}))$ in Theorem~\ref{thm:integerprogram} is tight.
\end{conjecture}

 Informally, the conjecture says that the secant variety has the maximal dimension possible given the sparsity pattern of its parameterization \eqref{eq:paramm}. This has been verified in many cases.

\begin{example} Let $n=4$, $d=12$. The inclusion
$\sigma_r(\mathcal{M}_{4,12})  \subset \mathbb{P}^{454}$
is strict for $r \leq 11$.
The dimensions are
 $47,91,135,175,215,255,291,$ $327,363,399,431$. This was
found correctly by  Theorem~\ref{thm:integerprogram}.
Compare this to the sequence
$\,47,91,135,179,223,267,311,355,399,443, 454$, which is
the upper bound  $\,{\rm min} \{44r + 3,454\}\,$ for ${\rm dim}\bigl( \sigma_r(\mathcal{M}_{4,12}) \bigr)$
given in Proposition \ref{prop:expdim}.
\end{example}

The question of finding the dimension is equally intriguing if we
replace the parameter $d$ by one specific partition $\lambda \vdash d$. 
Of particular interest is the partition $ \lambda = (1,1,1,\ldots,1) = (1^d)$. 
This toric variety has dimension $n-1$, and hence we have the trivial upper bound
\begin{equation}
 {\rm dim} \bigl( \sigma_r( \mathcal{M}_{n,(1^d)}) \bigr) \,\, 
\leq \,\, r(n-1) + r-1 \, = \,nr - 1. 
\end{equation}
Based on extensive computations, we conjecture that equality holds outside the matrix case:

\begin{conjecture}\label{conj:hypersimplex-dim}
Secant varieties of hypersimplices, other than the second hypersimplex, have the expected dimension.
In symbols, if $\, 3 \leq d \leq n-3\,$ then
$\,{\rm dim} \bigl( \sigma_r( \mathcal{M}_{n,(1^d)}) \bigr) = nr - 1 $.
\end{conjecture}

Theorem~5.1 in \cite{ZK} implies $\sigma_r(\mathcal{M}_{n, (1^d)})$ is strongly identifiable for $r \lesssim n^{\lfloor (d-1)/2 \rfloor}$.  
In particular, the secant variety has the expected dimension.  
Our next result is that the secant variety also has the expected dimension if $r \lesssim n^{d-2}$.  The proof relies on tropical geometry \nolinebreak \cite{Tropical}.

\begin{theorem}\label{thm:tropical}
The secant variety of the hypersimplex has the expected dimension if
\begin{equation} \label{eq:bound}
\left( 1 \, + \, d (n-d) \, + \, \binom{d}{2} \binom{n-d}{2} \right) r \quad \leq \quad \binom{n}{d}.
\end{equation}
\end{theorem}
\begin{proof}
Assume \eqref{eq:bound} holds.  Let $S = \Delta(n,d) \cap \mathbb{Z}^n$. 
From \cite[Lemma~3.8]{Tropical} and \cite[Corollary~3.2]{Tropical}, it suffices to show that there exist points $v_1, \ldots, v_r \in \mathbb{R}^n$ such that each of the Voronoi cells
$$
\textup{Vor}_i(v) := \{ \alpha \in S : \| \alpha - v_i \|_2 < \| \alpha - v_j \|_2 \textup{ for all } i \neq j \}
$$
spans an affine space of dimension $n-1$ inside $\mathbb{R}^n$.  
To argue this, choose $v_1 \in S$ arbitrarily and set $N_1 = \{ \alpha \in S : \| \alpha - v_1 \|_2 \leq 2 \}$.  Next choose $v_2 \in S \backslash N_1$ arbitarily and set $N_2 = \{ \alpha \in S : \| \alpha - v_2 \|_2 \leq 2 \}$.  Next choose $v_3 \in S \setminus \left(N_1 \cup N_2\right)$ arbitrarily and define $N_3$.  We continue until $S \setminus \left(N_1 \cup N_2 \cup \ldots \right) = \emptyset$. 
Note that the parenthesized sum on the left-hand side of \eqref{eq:bound} equals the size of each set $N_i$, while the right-hand side gives the size of $S$.  
Thus, \eqref{eq:bound} guarantees that we choose at least $r$ points in $S$.  Furthermore, these points
differ pairwise in at least $3$ coordinates by construction.
So, the $i$th Voronoi cell contains all elements of $S$ that differ from $v_i$ in at most one coordinate.  That is, it contains $v_i$ and all vertices in the hypersimplex adjacent to $v_i$.  Hence $\textup{Vor}_i(v)$ has the same affine span as $\Delta(n,d)$.    
\end{proof}

\section{Implicitization}
\label{sec6}

We verified the dimensions in Section \ref{sec5} with numerical methods for
fairly large instances, by computing the rank of the
Jacobian matrix of the parametrization (\ref{eq:param}).
For this we employed {\tt Maple},  {\tt Julia},  and the numerical {\tt Macaulay2} package in \cite{CK}.
We found it much more difficult to solve the implicitization problem,
that is, to compute the defining polynomials of our moment 
varieties.
The pentad (\ref{eq:pentad}) suggests that such polynomials can be quite \nolinebreak interesting.
In this section we offer more examples of equations, along with the degrees for our varieties. The code used in our computations is available on MathRepo.\footnote{\url{https://mathrepo.mis.mpg.de/MomentVarietiesForMixturesOfProducts}}

\begin{remark} \label{rem:idrepara}
It is preferable to work with birational parameterizations when numerically computing the degree of a variety \cite{BTW}.
However the map (\ref{eq:param}) is $d$-to-$1$: if $\omega$ is a primitive $d$th root of unity then
we can replace $\mu_{ki}$ by $\mu_{ki} \,\omega^i$ without changing 
 $m_{i_1 i_2 \cdots i_n}$. 
This implies that the  map (\ref{eq:paramm}) has fibers of size at least $r! \, d^r$.
We set $\mu_{2,1} = 1$ and let $\mu_{1,0}$ be an unknown to make (\ref{eq:param}) into a birational parameterization of $\mathcal{M}_{n,d}$.
Likewise, we turn (\ref{eq:paramm})
into a parametrization of $\sigma_r(\mathcal{M}_{n,d})$ that is expected to be $r!$-to-$1$, by
setting $\mu^{(j)}_{2,1}$ to $1$ and
using unknowns for $ \mu^{(j)}_{1,0}$.
\end{remark}

Let us now present a case study for implicitization, focused on the
hypersimplex $\Delta(6,3)$.

\begin{example}[$n=6,d=3$] 
The $5$-dimensional toric variety $\mathcal{M}_{6,(111)}$ lives in
$\mathbb{P}^{19}$, and it has degree $A(6,3)=66$ by Remark \ref{rmk:eulerian}.
Its  toric ideal is minimally generated by $69$ binomial quadrics.
These quadrics are the  $2 \times 2$ minors 
that are visible 
(i.e.~do not involve any stars)
in the following masked Hankel matrix:
\setcounter{MaxMatrixCols}{20}
\begin{align*}
\begin{small}
\hspace{-4em}\left[\begin{matrix}
   \star & \star & \star & \star & \star & \! m_{123} & \! m_{124} & \! m_{125} \\
    \star & \! m_{123} & \! m_{124} & \! m_{125} & \! m_{126} & \star & \star & \star \\
    \! m_{123} & \star & \! m_{134} & \! m_{135} & \! m_{136} & \star & \! m_{234} & \! m_{235} \\
    \! m_{124} & \! m_{134} & \star & \! m_{145} & \! m_{146} & \! m_{234} & \star & \! m_{245}\\
    \! m_{125} & \! m_{135} & \! m_{145} & \star & \! m_{156} & \! m_{235} & \! m_{245} & \star \\
    \! m_{126} & \! m_{136} & \! m_{146} & \! m_{156} & \star & \! m_{236} & \! m_{246} & \! m_{256} 
\end{matrix}\right.\end{small}\\ 
\end{align*}
\vspace{-2.5em}
\begin{align*}
&\qquad\qquad\qquad
\begin{small}\left.\begin{matrix}
   m_{126} &  \!  m_{134} & \! m_{135} & \! m_{136} & \! m_{145} & \! m_{146} & \! m_{156} \\
   \star & \! m_{234} & \! m_{235} & \! m_{236} & \! m_{245} & \! m_{246} & \! m_{256} \\
   m_{236} & \star & \star & \star & \! m_{345} & \! m_{346} & \! m_{356} \\
   m_{246} & \! \star & \! m_{345} & \! m_{346} & \star & \star & \! m_{456} \\
    m_{256} & \! m_{345} & \star & \! m_{356} & \star & \! m_{456} & \star \\
   \star & \! m_{346} & \! m_{356} & \star & \! m_{456} & \star & \star 
\end{matrix}\right].
\end{small}
\end{align*}

The rows are labeled by  $i \in \{1,2,\ldots,6\}$,
and the  columns by pairs $\{j,k\}$ of such indices.
The entry is $m_{ijk}$ if these are disjoint, and it is $\star$ otherwise.
Here we use
$\,m_{123} = m_{111000}, \,m_{124} = m_{110100}, \ldots, \, m_{456} = m_{000111}$.
Our toric ideal is generated by all $2 \times 2$ minors without $\star$.

The $6 \times 15$ matrix has twenty $3 \times 3$ minors without $\star$,
and these vanish on $\sigma_2( \mathcal{M}_{6,(111)})$.
In addition to these cubics, the ideal contains $12$ pentads (\ref{eq:pentad}),
one for each facet $\Delta(5,2)$ of the hypersimplex $\Delta(6,3)$.
Our ideal for $r=2$ is generated by these $20$ cubics and $12$ quintics.
Numerical degree computations using Remark  \ref{rem:idrepara}
with {\tt HomotopyContinuation.jl} \cite{BreTim}  reveal 
\begin{equation}
\label{eq:numdegree}
 {\rm deg}(\sigma_2 (\mathcal{M}_{6,(111)})) =  465\,
 {\rm and} \,
  {\rm deg}(\sigma_3 (\mathcal{M}_{6,(111)})) =  80.
\end{equation}
Symbolic computations for
 $r=3$ are challenging. Our secant variety has
codimension $2$ in $\mathbb{P}^{19}$. There are no quadrics or cubics vanishing on
 $\sigma_3 (\mathcal{M}_{6,(111)})$, but 
 there is a unique quartic:
\begin{equation}
\label{eq:quadset} 
 \begin{small} \!\!\!\!\!\!\begin{matrix}
     \,\,m_{123} m_{145} m_{246} m_{356} {-} m_{123} m_{145} m_{256} m_{346}
  {-} m_{123} m_{146} m_{245} m_{356} \\ {+} m_{123} m_{146} m_{256} m_{345}
  {+} m_{123} m_{156} m_{245} m_{346}  {-} m_{123} m_{156} m_{246} m_{345} \\
  {-} m_{124} m_{135} m_{236} m_{456} {+} m_{124} m_{135} m_{256} m_{346}
  {+} m_{124} m_{136} m_{235} m_{456} \\ {-} m_{124} m_{136} m_{256} m_{345} 
  {-} m_{124} m_{156} m_{235} m_{346} {+} m_{124} m_{156} m_{236} m_{345} \\
  {+} m_{125} m_{134} m_{236} m_{456} {-} m_{125} m_{134} m_{246} m_{356}
  {-} m_{125} m_{136} m_{234} m_{456} \\ {+} m_{125} m_{136} m_{246} m_{345}
  {+} m_{125} m_{146} m_{234} m_{356} {-} m_{125} m_{146} m_{236} m_{345} \\
  {-} m_{126} m_{134} m_{235} m_{456} {+} m_{126} m_{134} m_{245} m_{356} 
  {+} m_{126} m_{135} m_{234} m_{456} \\ {-} m_{126} m_{135} m_{245} m_{346}
  {-} m_{126} m_{145} m_{234} m_{356} {+} m_{126} m_{145} m_{235} m_{346} \\
  {+} m_{134} m_{156} m_{235} m_{246} {-} m_{134} m_{156} m_{236} m_{245}
  {-} m_{135} m_{146} m_{234} m_{256} \\ {+} m_{135} m_{146} m_{236} m_{245}
  {+} m_{136} m_{145} m_{234} m_{256} {-} m_{136} m_{145} m_{235} m_{246}.
\end{matrix} \end{small}
\end{equation}
Note the beautiful combinatorics in this polynomial: the role of the $5$-cycle for the
pentad is now played by the {\em quadrilateral set}, i.e.~the six intersection points
of four lines in the plane.
\end{example}

We conclude this article with the smallest non-trivial
secant varieties. Here ``non-trivial'' means $r \geq 2$,
the variety does not fill its ambient projective space,
and the ambient dimension is as small as possible.
The next two results feature all 
cases where $\binom{n+d-1}{d} \leq 50$.
The list consists of $(r,n,d) = (2,5,3)$ 
from Example \ref{ex:M53toric} 
and  $(r,n,d)=(2,4,4)$ from
Example~\ref{ex:M44toric}.
We state these as propositions because they 
represent case studies that are
of independent interest for 
experimental mathematics, especially in
the ubiquitous setting of tensor decompositions.

\begin{proposition} \label{prop:M53}
The secant variety $\sigma_2(\mathcal{M}_{5,3})$ has dimension $24$ and degree $3225$
in $\mathbb{P}^{34}$. Its prime ideal is generated by $313$ polynomials,
namely $10$ cubics, $283$ quintics, $10$ sextics and $10$ septics.
These ideal generators are obtained by elimination from
the ideal of $3 \times 3$ minors of the $5 \times 15$~matrix 
\setcounter{MaxMatrixCols}{20}
\begin{align}\label{eq:hankel53}
\begin{small}
\hspace{0em}\left[\begin{matrix}
    a_{23} &  a_{24} &  a_{25} &  a_{34} &  a_{35} &  a_{45} & \star & \star  &\star &  \star  \\
    a_{13} &  a_{14} & a_{15} & \star & \star & \star &  a_{34} & a_{35} & a_{45}&  \star\\
    a_{12} & \star & \star & a_{14} & a_{15} & \star & a_{24} & a_{25} & \star & a_{45} \\
    \star & a_{12} & \star & a_{13} & \star & a_{15} & a_{23} & \star & a_{25} & a_{35}\\
    \star  & \star & a_{12} & \star & a_{13} & a_{14} & \star & a_{23} & a_{24} & a_{34} 
\end{matrix}\right.\end{small}
\end{align}
\begin{align*}
&\qquad\qquad
\qquad\qquad\qquad\quad
\begin{small}\left.\begin{matrix}
\star &  b_{21} &  b_{31} &  b_{41} &  b_{51} \\
b_{12} &  \star &  b_{32} &  b_{42} &  b_{52} \\
 b_{13} &  b_{23} &  \star &  b_{43} &  b_{53} \\
 b_{14} & b_{24} & b_{34} & \star & b_{54} \\
b_{15} & b_{25} & b_{35} & b_{45} & \star
\end{matrix}\right].
\end{small}
\end{align*}
\end{proposition}

Proposition \ref{prop:M53} is
important in that it
displays
a general technique of obtaining equations for
varieties of low rank 
structured symmetric tensors
from masked Hankel matrices.

\begin{proof}[Notation and Proof]
The visible
entries in the masked matrix (\ref{eq:hankel53}) are $30$ of the $35$ moments $m_{i_1i_2i_3i_4i_5}$.
Here, the $10$ moments for $\lambda = (111) $ are denoted
$a_{12} = m_{00111},
 a_{13} = m_{01011},
\ldots
a_{35} = m_{11010}, 
a_{45} = m_{11100}$, and
the $20$ moments for $\lambda = (12)$ are
 $\,
 b_{12} = m_{21000}, b_{13} = m_{20100},
\ldots, b_{21} = m_{12000}, \ldots,
 b_{53} = m_{00102}, b_{54} = m_{00012} $.
 The $25$ stars are distinct new unknowns, and these are being eliminated.
 The matrix contains ten $3 \times 3$-submatrices with no stars.
 Their determinants are the ten cubics
 mentioned
 in Proposition   \ref{prop:M53}.
 
 The ideal of $\sigma_2(\mathcal{M}_{5,3})$
 is homogeneous in the bigrading 
 given by $a$ and $b$.
   Among the generators, we find
    $1, 55, 110,$ $ 90, 27$ quintics of bidegrees $(5,0),(3,2),(2,3),(1,4),(0,5)$.
The quintic of bidegree $(5,0)$ is the pentad of the symmetric $5 \times 5$-matrix $(a_{ij})$.
One of the quintics of bidegree $(2,3)$~is
$$ \begin{small} \begin{matrix}
a_{13} a_{45} b_{25} b_{41} b_{53}{-}a_{13} a_{45}  b_{25} b_{43} b_{51}
{-}a_{14} a_{34} b_{21} b_{43} b_{54}\\
\!\!\!{+}a_{14} a_{34} b_{23} b_{41} b_{54} 
{-}a_{14} a_{35} b_{23} b_{45} b_{51}{+}a_{14} a_{35} b_{25} b_{43} b_{51}  
\\
\!\!\!+a_{14} a_{45} b_{24} b_{45} b_{51}{-}a_{14} a_{45} b_{25} b_{41} b_{54}
{+}a_{15} a_{34} b_{21} b_{45} b_{53}\\
\!\!\!{-}a_{15} a_{34} b_{25} b_{41} b_{53}
{-}a_{34} a_{45}  b_{24} b_{45} b_{53}{+}a_{34} a_{45} b_{25} b_{43} b_{54}.
\end{matrix} \end{small}
$$
The ten sextics have bidegrees $(4,2)$ and $(0,6)$, five each.
All ten septics have bidegree $(3,4)$.
 The $27+5$ generators of bidegrees $(0,5)$ and $(0,6)$
 generate the prime ideal of $\sigma_2(\mathcal{M}_{5,(21)})$.
 They arise from the $5 \times 5$ matrix $(b_{ij})$
   by eliminating the diagonal.
 The degree $3225$ was first found
  numerically, and later confirmed
 symbolically by {\tt Macaulay2}.
 \end{proof}

Our final result concerns
tensors of format
$4 \times 4 \times 4 \times 4$.

\begin{proposition} \label{prop:M44}
The secant variety $\sigma_2(\mathcal{M}_{4,4})$ has dimension $27$ and degree $8650$
in $\mathbb{P}^{34}$. Its prime ideal  has only three minimal generators in degrees at most six.
These are  the cubics 

\begin{equation*}
 {\rm det}
\begin{bmatrix}
     m_{2200} & m_{2110} & m_{2020} \\
     m_{1201} & m_{1111} & m_{1021} \\ 
     m_{0202} & m_{0112} & m_{0022} \\ 
\end{bmatrix}\! , \,\,
{\rm det}
\begin{bmatrix}
     m_{2200} & m_{2101} & m_{2002} \\ 
     m_{1210} & m_{1111} & m_{1012} \\ 
     m_{0220} & m_{0121} & m_{0022} \\ 
\end{bmatrix} \! ,\,\,
\end{equation*}
\begin{equation}\label{eq:threecubics}
{\rm det}
\begin{bmatrix}
     m_{2020} & m_{2011} & m_{2002} \\ 
     m_{1120} & m_{1111} & m_{1102} \\ 
     m_{0220} & m_{0211} & m_{0202}
\end{bmatrix} \! .
\end{equation}
\end{proposition}
Proposition \ref{prop:M44} is proved by direct computation.
 The degree $8650$ was found with {\tt HomotopyContinuation.jl} using the method in \cite{BTW}.
The absence of minimal generators in degrees $4,5,6$ was verified by solving the
linear equations for each  $A$-degree in that range. Each solution was found to be
in the ideal (\ref{eq:threecubics}).  At present we know of no ideal generators for 
$\sigma_2(\mathcal{M}_{4,4})$ that involve the 
moments $m_{3100},m_{1300}, \ldots,m_{0013}$.
What is the smallest degree in which 
we can find such generators? \\

\bigskip

\noindent\textbf{Acknowledgements.} 
The authors are grateful to Jan Draisma, Marc H{\"a}rk{\"o}nen, Simon Telen, and Yifan Zhang for helpful conversations.

\bigskip
\footnotesize
\noindent {\bf Authors' addresses:}

\smallskip

\noindent Yulia Alexandr, UC Berkeley \hfill \href{mailto:yulia@math.berkeley.edu}{\texttt{yulia@math.berkeley.edu}}

\noindent Joe Kileel, UT Austin \hfill \href{mailto:jkileel@math.utexas.edu}{\texttt{jkileel@math.utexas.edu}}

\noindent Bernd Sturmfels, MPI-MiS Leipzig  and UC Berkeley \hfill \href{mailto:bernd@mis.mpg.de}{\texttt{bernd@mis.mpg.de}}


\bigskip
\bigskip


\begin{thebibliography}{10}
\bibitem{AFS} C.~Am\'endola, J.-C.~Faug\`ere and B.~Sturmfels:
{\em Moment varieties of Gaussian mixtures},
Journal of Algebraic Statistics {\bf 7} (2016) 14--28.

\bibitem{BP} M.~Beck and D.~Pixton: {\em The Ehrhart Polynomial of the Birkhoff Polytope}, Discrete \& Computational Geometry {\bf 30} (2003) 623--637.

\bibitem{BP-arxiv} M.~Beck and D.~Pixton: {\em The volume of the 10th Birkhoff polytope}, {\tt arXiv:math/0305332}.

 
 \bibitem{Bil} P.~Billingsley:
 {\em Probability and measure}, Wiley-Interscience, New York, 1995.
  
  \bibitem{BreTim}
  P.~Breiding and S.~Timme:
  {\em HomotopyContinuation.jl: A package for homotopy continuation in Julia},
 International Congress on Mathematical Software, Springer, 2018, pp.~458--465.
\bibitem{BD}
A.~Brouwer and J.~Draisma: 
{\em Equivariant Gr\"obner bases and the Gaussian two-factor model},
Mathematics of Computation {\bf 80} (2011) 1123--1133. 

\bibitem{BTW}
L.~Brustenga, I. Moncus\'i, S.~Timme and M.~Weinstein:
{\em 96120: The degree of the linear orbit of a cubic surface},
Matematiche (Catania) {\bf 75} (2020) 425--437. 

\bibitem{CK}
J.~Chen and J.~Kileel: {\em Numerical implicitization}, Journal of Software for Algebra and Geometry {\bf 9} (2019) 55--63.

\bibitem{DS} P.~Diaconis and B.~Sturmfels:
{\em Algebraic algorithms for sampling from conditional distributions},
Annals of Statistics {\bf 26} (1998) 363--397.

\bibitem{Tropical} J.~Draisma: {\em A tropical approach to secant dimensions}, Journal of Pure and Applied Algebra {\bf 212} (2008) 349--363.

\bibitem{Draisma} J.~Draisma: {\em Finiteness for the $k$-factor model and chirality varieties},
  Advances in Mathematics {\bf 223} (2010) 243--256.

\bibitem{DEFM} J.~Draisma, R.~Eggermont, A.~Farooq and L.~Meier:
  {\em Image closure of symmetric wide-matrix varieties}, {\tt arXiv:2212.12458}.
  
\bibitem{DEKL} J.~Draisma, R.~Eggermont, R.~Krone and A.~Leykin:
{\em Noetherianity for infinite-dimensional toric varieties},
Algebra Number Theory {\bf 9} (2015) 1857--1880. 

\bibitem{DSS} M.~Drton, B.~Sturmfels and S.~Sullivant:
{\em Algebraic factor analysis: Tetrads, pentads and beyond},
Probability Theory and Related Fields {\bf 138} (2007) 463--493.

\bibitem{KSS}  K.~Kohn, B.~Shapiro and B.~Sturmfels:
{\em Moment varieties of measures on polytopes}, 
Annali della Scuola Normale Superiore di Pisa {\bf 21} (2020) 739--770.

\bibitem{Liu} R.~Liu: {\em Laurent polynomials, Eulerian numbers, and Bernstein's theorem},
Journal~of~Combinatorial~Theory, Series~A {\bf 124} (2014) 244--250. 

\bibitem{OH} H.~Ohsugi, Hidefumi and T.~Hibi: {\em Toric ideals generated by quadratic binomials}, Journal~of~Algebra {\bf 218} (1999) 509--527.

\bibitem{GBCP} B.~Sturmfels: {\em Gr\"obner bases and convex polytopes}.
American Mathematical Society, University Lectures Series, No 8, Providence, Rhode Island, 1996.

\bibitem{Sul} S.~Sullivant: {\em
Algebraic statistics}, Graduate Studies in Mathematics, 194,
American Mathematical Society, Providence, RI, 2018.

\bibitem{YOT} T.~Yamaguchi, M.~Ogawa and A.~Takemura:
{\em Markov degree of the Birkhoff model}, Journal of Algebraic Combinatorics {\bf 40} (2014) 293--311.

\bibitem{ZK}  Y.~Zhang and J.~Kileel: {\em
Moment estimation for nonparametric mixture models through implicit tensor decomposition},
{\tt arXiv:2210.14386}.

\end{thebibliography}
\end{document}